\documentclass[10pt]{amsart}
\usepackage{amscd}
\usepackage{amsmath}
\usepackage{amssymb}
\usepackage{latexsym}
\usepackage{lscape}
\usepackage{xypic}
\usepackage{hyperref} 
\usepackage{color}

\newtheorem{thm}{Theorem}[section]
\newtheorem{prop}[thm]{Proposition}
\newtheorem{lem}[thm]{Lemma}

\newtheorem{lem-def}[thm]{Lemma-Definition}
\newtheorem{cor}[thm]{Corollary}
\theoremstyle{remark}

\newtheorem{rmk}{Remark}[section]
\theoremstyle{definition}

\newtheorem{ques}[thm]{Question}

\numberwithin{equation}{section}

\newcommand{\nc}{\newcommand}
\nc{\on}{\operatorname}

\newcommand{\frakb}{{\mathfrak b}}

\newcommand{\frakg}{{\mathfrak g}}

\newcommand{\frakn}{{\mathfrak n}}

\newcommand{\frakt}{{\mathfrak t}}

\newcommand{\bA}{{\mathbb A}}

\newcommand{\bC}{{\mathbb C}}
\newcommand{\bD}{{\mathbb D}}

\newcommand{\bG}{{\mathbb G}}

\newcommand{\bP}{{\mathbb P}}

\newcommand{\bR}{{\mathbb R}}

\newcommand{\bZ}{{\mathbb Z}}

\newcommand{\calF}{{\mathcal F}}

\newcommand{\calL}{{\mathcal L}}

\newcommand{\calO}{{\mathcal O}}

\newcommand{\calT}{{\mathcal T}}
\newcommand{\calU}{{\mathcal U}}

\nc{\al}{{\alpha}} \nc{\be}{{\beta}} \nc{\ga}{{\gamma}}
\nc{\ve}{{\varepsilon}} \nc{\Ga}{{\Gamma}} \nc{\la}{{\lambda}}
\nc{\La}{{\Lambda}}

\nc{\ad}{{\on{ad}}}

\nc{\aff}{{\on{aff}}}
\nc{\Aff}{{\mathbf{Aff}}}
\newcommand{\Aut}{{\on{Aut}}}
\nc{\Bun}{{\on{Bun}}}

\nc{\der}{{\on{der}}}
\newcommand{\Der}{{\on{Der}}}
\nc{\diag}{{\on{diag}}}
\newcommand{\End}{{\on{End}}}
\nc{\Fl}{{\calF\ell}}

\newcommand{\Hom}{{\on{Hom}}}
\nc{\Hol}{{\on{Hol}}}

\nc{\Id}{{\on{Id}}}

\nc{\Ind}{{\on{Ind}}}

\newcommand{\Res}{{\on{Res}}}

\nc{\res}{{\on{res}}}

\nc{\tr}{{\on{tr}}}

\newcommand{\GL}{{\on{GL}}}

\nc{\GSp}{{\on{GSp}}} \nc{\GU}{{\on{GU}}} \nc{\SL}{{\on{SL}}}
\nc{\SU}{{\on{SU}}} \nc{\SO}{{\on{SO}}}
\nc{\Sol}{{\on{Sol}}}
\nc{\Ann}{{\on{Ann}}}
\nc{\codim}{{\on{codim}}}

\nc{\four}{{\calF our}}

\newcommand{\fg}{\mathfrak{g}}

\newcommand{\ft}{\mathfrak{t}}

\newcommand{\HH}{\mathbb{H}}

\newcommand{\CC}{\mathbb{C}}

\newcommand{\C}{\CC}
\newcommand{\PP}{\mathbb{P}}
\renewcommand{\P}{\PP}
\newcommand{\G}{\mathbb{G}}
\newcommand{\ZZ}{\mathbb Z}
\newcommand{\Z}{\ZZ}
\newcommand{\into}{\hookrightarrow}

\newcommand{\ra}{\rightarrow}

\newcommand{\br}{\buildrel}

\newcommand{\cF}{{\mathcal F}}

\newcommand{\cN}{{\mathcal N}}
\newcommand{\cO}{{\mathcal O}}

\newcommand{\cY}{{\mathcal Y}}

\def\question#1{{}}

\newcommand{\quash}[1]{}  

\author{An Huang, Bong H. Lian, Shing-Tung Yau, and Xinwen Zhu}

\title{Chain Integral Solutions to Tautological Systems}
\date{August 3, 2015}

\begin{document}

\maketitle
\begin{abstract}
We give a new geometrical interpretation of the local analytic solutions to a differential system, which we call a tautological system $\tau$, arising from the universal family of Calabi-Yau hypersurfaces $Y_a$ in a $G$-variety $X$ of dimension $n$. First, we construct a natural topological correspondence between relative cycles in $H_n(X-Y_a,\cup D-Y_a)$ bounded by the union of $G$-invariant divisors $\cup D$ in $X$ to the solution sheaf of $\tau$, in the form of chain integrals. Applying this to a toric variety with torus action, we show that  in addition to the period integrals over cycles in $Y_a$, the new chain integrals generate the full solution sheaf of a GKZ system. This extends an earlier result for hypersurfaces in a projective homogeneous variety, whereby the chains are cycles \cite{BHLSY,HLZ}. In light of this result, the mixed Hodge structure of the solution sheaf is now seen as the MHS of $H_n(X-Y_a,\cup D-Y_a)$. In addition, we generalize the result on chain integral solutions to the case of general type hypersurfaces. This chain integral correspondence can also be seen as the Riemann-Hilbert correspondence in one homological degree.  Finally, we consider interesting cases in which the chain integral correspondence possibly fails to be bijective.
\end{abstract}

\tableofcontents
\baselineskip=16pt plus 1pt minus 1pt
\parskip=\baselineskip

\pagenumbering{arabic}
\addtocounter{page}{0}
\markboth{\SMALL
A. Huang, B.H. Lian, S.-T. Yau, and X. Zhu}
{\SMALL Chain Integral Solutions to Tautological Systems}


\section{Introduction}\label{Intro}

{\it Throughout this paper, we shall follow closely the notations introduced in \cite{HLZ}.} 
Let $G$ be a connected algebraic group over a field $k$ of characteristic zero.
Let $X$ be a projective $G$-variety of dimension $n$, and let $\calL$ be a very
ample $G$-linearized invertible sheaf over $X$ which gives rise to a
$G$-equivariant embedding
\[X\to \bP(V),\]
where $V=\Gamma(X,\calL)^\vee$.
Let $r=\dim V$. We assume that the action of $G$ on $X$ is locally effective, i.e. $\ker(G\to\Aut(X))$ is finite. Let $\bG_m$ be the multiplicative group acting on $V$ by homotheties. Let $\hat{G}=G\times\bG_m$, whose Lie algebra is $\hat{\frakg}=\frakg\oplus ke$, where $e$ acts on $V$ by identity.
 We denote by $Z:\hat{G}\to \GL(V)$ the corresponding group representation, and $Z:\hat{\frakg}\to \End(V)$ the corresponding Lie algebra representation. 
Note that under our assumptions, $Z:\hat{\frakg}\to \End(V)$ is injective.

Let $\hat{\imath}:\hat{X}\subset V$ be the cone of $X$, defined by
the ideal $I(\hat{X})$. Let $\beta:\hat{\frakg}\to k$ be a Lie algebra
homomorphism. Then a {\it tautological system} as defined in \cite{LSY}\cite{LY} is a cyclic $D$-module on $V^\vee$ given by
\[\tau\equiv\tau(G,X,\calL,\beta)=D_{V^\vee}/D_{V^\vee}J(\hat{X})+D_{V^\vee}(Z(\xi)+\beta(\xi), \xi\in \hat{\frakg}),\]
where
$$J(\hat{X})=\{\widehat{D}\mid D\in I(\hat{X})\}$$
is the ideal of the commutative subalgebra $\bC[\partial]\subset
D_{V^\vee}$ obtained by the Fourier transform of $I(\hat{X})$ (see \cite[\S A]{HLZ} for the review on the Fourier transform).

Given a basis $\{a_i\}$ of $V$, we have $Z(\xi)=\sum_{ij}\xi_{ij}a_i\partial_{a_j}$, where $(\xi_{ij})$ is the matrix representing $\xi$ in the basis. Since the $a_i$ are also linear coordinates on $V^\vee$, we can view $Z(\xi)\in\Der k[V^\vee]\subset D_{V^\vee}$. In particular, the identity operator $Z(e)\in\End V$ becomes the Euler vector field on $V^\vee$. 

We briefly recall the main geometrical context that motivates our study of tautological systems. Let $X'$ be a compact complex manifold (not necessarily algebraic), such that the complete linear system of anticanonical divisors in $X'$ is base point free.
Let $\pi:\cY\ra B:=\Gamma(X',\omega_{X'}^{-1})_{sm}$ be the universal family of smooth CY hyperplane sections $Y_a\subset X'$, and let $\HH^{top}$ be the Hodge bundle over $B$ whose fiber at $a\in B$ is the line $\Gamma(Y_a,\omega_{Y_a})\subset H^{n-1}(Y_a)$, where $n=\dim X'$. In \cite{LY}, the period integrals of this family are constructed by giving a canonical trivialization of $\HH^{top}$. Let $\Pi=\Pi(X')$ be the period sheaf of this family, i.e. the locally constant sheaf generated by the period integrals (Definition 1.1 \cite{LY}.)

Integral solutions to holonomic differential systems go back to the classical theory of the Gauss hypergeometric equation in the form of the so-called Euler integrals. Many generalizations have since been found over the centuries. One notable class was the vast generalizations given by the celebrated GKZ hypergeometric systems \cite{GKZ1990,Adolphson} associated to algebraic tori and their rational modules. Euler type integral solutions to these systems have been constructed, and are integrals of multivalued meromorphic differential forms over `formal cycles', namely they are homology classes on the complement of an affine hypersurface in an algebraic torus with coefficient in a local system. This construction has also  been generalized later to hypergeometric systems associated to reductive algebraic groups and their rational modules \cite{Kapranov1997}. 

On the other hand, in recent decades period integrals of projective varieties have become central to the study of mirror symmetry and Hodge theory. As it is well-known, for the universal family of CY hypersurfaces in a given toric variety $X$, the GKZ system $\tau$ whose solutions include period integrals of the family, is never complete in the sense that its solution sheaf is always strictly larger than the period sheaf. While the latter is by construction geometrical in nature, physicists have conjectured that the larger solution sheaf too has a purely geometrical origin. In fact, they have shown in some examples that the solutions to $\tau$ in this case are integrals over topological chains with certain boundary conditions \cite{ADJ}, and they call these solutions `semi-periods' of the CY family.  In addition, period integrals over relative cycles have also arisen in another context in mirror symmetry, namely in the theory of open string theory \cite{KW2008,JS2009}. Here the relative cycles are chains bounded by certain distinguished algebraic curves (or `D-branes') in a CY threefold (see \cite{LLY2010} and for details), and they are the basic ingredients for enumerating open Gromov-Witten invariants in this setting.

In this paper, we show that the so-called semi-periods in physics are nothing but integrals over relative cycles with boundary on the $G$-invariant canonical divisor, and we do so for CY hyperplane sections in a general $G$-variety $X$. We also extend this result to general type hyperplane sections. In addition, we show that the chain integrals we have constructed do in fact exhaust all solutions when $X$ is a toric variety. Therefore, these chain integrals may also be viewed as a geometrical realization of the solutions to a GKZ system, as periods associated to relative cycles for families of algebraic varieties. We note that the chain integrals are defined here for families of projective varieties, and are therefore \`a priori different from the Euler type integrals in the GKZ theory, since the two are integrals over classes in different homology groups. But since both types of integrals solve the same differential system in the case in question, it would be very interesting to find a direct correspondence between the two constructions. This will be deferred to a future investigation.

We now return to the main geometrical set up of this paper. We shall assume that $X$ is an $n$-dimensional finite-orbit Fano smooth $G$-variety.  We denote by $\cup D$ the union of all $G$-invariant divisors in $X$ (which may be empty). Let $V=\Gamma(X,\omega_{X}^{-1})^\vee$, and we identify $X$ with the image of the natural map $X\to\P(V)$, and put $\calL=\calO_X(1)$. 
We shall consider the tautological system $\tau=\tau(X,G,\calL,\beta)$ in two important settings:\\
(1) $\calL=\omega_X^{-1}$, and $\beta=\beta_0$, where $\beta_0(\fg):=0$ and $\beta_0(e):=1$, as in the setting of CY hyperplane sections above; and more generally \\
(2) $\calL$ is any very ample line bundle such that $\calL\otimes\omega_X$ is base point free, and $\beta=\beta_0$. This case corresponds to hyperplane sections of general type. \\
Some of the results on toric varieties also hold under much weaker conditions.

 Put 
$$D_{X,\beta}=(D_X\otimes k_\beta)\otimes_{U\frakg} k,$$
which is a $D_X$-module. Here we treat $\beta\equiv\beta|_{\frakg}$, and $k_\beta$ is the 1-dimensional $\frakg$-module given by the character $\beta$ (see \cite[\S A]{HLZ} for details on notations). 
This D-module will play an important role throughout the paper.

Here is a brief outline. We begin in \S\ref{general} with the construction of the `chain integral map', between relative cycles in $H_n(U_a,U_a\cap(\cup D))$ and local analytic solutions at an arbitrary point $a\in V^\vee$ to $\tau$. Here $U_a:=X-Y_a$. The rest of the paper is then devoted to studying this correspondence. Theorem \ref{middle relative homology} shows that the chain integral map is bijective when $X$ is a toric variety and $\tau$ is a GKZ system (i.e. the symmetry group is chosen to be the torus.) Theorem \ref{middle relative homology - general type} proves an analogous result for general type hyperplane sections. Our main tool here is a new description, Proposition \ref{Kummer local system}, of the D-module $D_{X,\beta}$, together with a previous description of $\tau$ given by the Riemann-Hilbert correspondence \cite{HLZ}. In \S\ref{concluding remarks}, we consider cases in which the chain integral map may fail to be bijective, and give a description of the kernel and cokernel of the map.

\section{Chain integral solutions to $\tau$}\label{general}

We begin with the following observation on the universal family of CY hyperplane sections in $X$.

\begin{prop}
For any relative cycle $C\in H_n(U_a,U_a\cap(\cup D))$, the chain integral $\int_C{\Omega\over f_a}$ is a solution to $\tau$.
\end{prop}
\begin{proof}
The proof will essentially be the same as in \cite[Thm. 8.8]{LY} in the case CY hypersurfaces, except for one crucial difference. Here $C$ plays the role of a cycle $\Gamma\in H_n(U_a,\C)$ there, which was automatically $G_0$-invariant ($G_0$ is the connected component of $G$), a fact used in \cite{LY} to argue that $\int_\Gamma{\Omega\over f_a}$ is $G_0$-invariant. In order to complete the proof here, it suffices to show the analogous statement that $\int_C{\Omega\over f_a}$ is $G_0$-invariant (although $C$ itself need not be so!)

By assumption, $C$ is an $n$-chain in $U_a$ bounded by the $G_0$-invariant divisor $\cup D$. Let $x\in\frakt$. For small $\varepsilon>0$, consider the chains
$$
 C_\varepsilon=\{e^{tx}c|c\in C,~t\in[0,\varepsilon]\},\hskip.2in C_\varepsilon'=\{e^{tx}c|c\in\partial C,~t\in[0,\varepsilon]\}.
$$
(Here we have abuse notations slightly by representing a chain by its image set in $X$, but its meaning as a chain should be clear in this context.) Then $C'$ is an $(n+1)$-chain with
$$
\partial C_\varepsilon=e^{\varepsilon x}C-C+C_\varepsilon'.
$$
Obviously, $\int_{\partial C_\varepsilon}{\Omega\over f}=0$. Since $\Omega/f_a$ is a holomorphic in $U_a$, its restriction to any divisor of $X$ is zero. In particular, since $C_\varepsilon'\subset \cup D$ , it follows that $\int_{C_\varepsilon'}{\Omega\over f_a}=0$ as well. Thus we conclude that 
$$
\int_{e^{\varepsilon x}C}{\Omega\over f_a}=\int_{C}{\Omega\over f_a}
$$
proving that the right side is $G_0$-invariant. 

The rest of the proof is the same as in  \cite[Thm. 8.8]{LY}.
\end{proof}

This shows that in general, we have a canonical `chain integral' map
\begin{equation}\label{chain integral map}
H_n(U_a,U_a\cap(\cup D))\ra \Hom_{D_{V^\vee}}(\tau,\cO_{V^\vee,a}^{an}),\hskip.2in C\mapsto\int_{C}{\Omega\over f_a}
\end{equation}
Note that this map extends the so-called cycle-to-period map \cite{LY}\cite{HLZ}:
$$
H_n(U_a,\C)\mapsto \Hom_{D_{V^\vee}}(\tau,\cO_{V^\vee,a}^{an}),\hskip.2in \Gamma\mapsto\int_{\Gamma}{\Omega\over f_a}.
$$
In other words, the cycle-to-period map factors through the natural map
$$
H_n(U_a,\C)\ra H_n(U_a,U_a\cap(\cup D))
$$
and the chain integral map.

\begin{ques} 
When is the \emph{chain integral map} \eqref{chain integral map} is an isomorphism? More generally, when is the same true for general type hyperplane sections?
\end{ques}
Two of our main results, Theorems \ref{middle relative homology} and \ref{middle relative homology - general type}, will show that the answer is affirmative when $X$ is a toric variety, i.e. $\tau$ is a GKZ system. One of the results in \cite{HLZ} shows the same is true for any projective homogeneous $G$-varieties as well (where $\cup D$ is empty). We will also discuss cases, including some examples, in which the isomorphism possibly fails, and describe the kernel and cokernel of the chain integral map in these cases.

\section{Chain integral solutions to GKZ systems}

We shall now work over the ground field $k=\C$. Put $T=\G_m^n$, let  $X$ be a $n$-dimensional smooth projective toric variety with respect to $G=T$, and fix a very ample line bundle $\calL$ over $X$.  Recall that
$$
V=\Gamma(X,\calL)^\vee,
$$
and let $a_i$ denote a basis of $V$, $a_i^\vee$ the dual basis, and let $f\equiv\sum_i a_i^\vee a_i:X\times V^\vee\ra\calL$ be the universal section of $\calL$. Note that in this set up, $\tau$ becomes a GKZ hypergeometric system \cite{GKZ1990}. 

Recall that in the present setting, the union $\cup D$ of all $T$-invariant divisors in $X$ is the anticanonical divisor of $X$. Let $i_{\cup D},j_{\cup D}\equiv j$ be respectively the closed and open embeddings of $\cup D,X-\cup D$ into $X$. Let $D$ is an irreducible component of $\cup D$, and $i_D,j_D$ be respectively the closed and open embeddings of $D,X-D$ into $X$. We begin with the following key lemma.

Our next result will be formulated for an arbitrary smooth toric variety $X$, {\it possibly incomplete}. We shall need to apply it to affine toric varieties in order to prove the main result Theorem \ref{middle relative homology}.
Let $X$ be an $n$-dimensional smooth toric variety, with the $T=\G_m^n$ action. Let $\Sigma\subset \bR^n$ be the fan associated to $X$, and $\{v_i\}\subset N=\bZ^n$ be the integral generators of the 1-cones of $\Sigma$. We say that $\alpha\in N^\vee_\bR$ has property (*) if 
$$(*)\hskip.2in \alpha(v_i)\neq 0,-1,-2,\ldots\mbox{ for every $v_i$}.$$
(Note that this is slightly different from the semi-nonresonance condition in \cite{GKZ1990, Adolphson}.)

\begin{prop}\label{Kummer local system}
Assume that $X$ is smooth, and $\alpha$ has property (*). Then 
$$D_{X,\alpha}:=(D_X\otimes\alpha)\otimes_{U\frakt}\bC\simeq j_+\calL_{\alpha}$$ 
where $j:\mathring{X}\equiv X-\cup D\to X$ is the open embedding of the open dense $T$-orbit, and $\calL_{\alpha}:= (D_{\mathring{X}}\otimes\alpha)\otimes_{U\frakt}\bC=j^!D_{X,\alpha}.$ 
\end{prop}
\begin{rmk}Note that $\calL_{\alpha}$ 
is a rank one local system on $\mathring{X}$. Under an identification $\mathring{X}\simeq T$, it is a character D-module on $T$, see (A.5) of \cite{HLZ}, usually called a Kummer local system.
\end{rmk}

This proposition follows from
\begin{lem}\label{!-fiber of DX vanishing}
Let $i:D=D_i\subset X\to X$ be a boundary divisor, corresponding to some $v_i$. If $\alpha(v_i)\neq 0,-1,-2,\ldots$, then $i^!D_{X,\alpha}=0$.
\end{lem}
\begin{proof}
By covering $X$ by affine open toric subvarieties, we can assume that $X\simeq \bA^r\times\bG_m^{s}$ with coordinates $\{z_1,\ldots,z_r,z_{r+1},\ldots,z_{r+s}\}$, $n=r+s$, and that $D=\{0\}\times\bA^{r-1}\times \bG_m^s$ is given by $z_1=0$. Then
\[D_{X,\alpha}=D_X/ \sum_{i=1}^nD_X(z_i\partial_i+ \alpha_i),\]
where $\alpha_i=\alpha(e_i')$, and a basis $\{e_i'\}$ of $\ft$ acts on the affine toric variety $X$ as $z_i\partial_i$. By the assumption, $\alpha_1\neq 0,-1,-2,\ldots,$ (and $\alpha_{2},\ldots,\alpha_{r+s}$ are irrelevant in this local argument.) 

Then by \cite[Thm. 7.4, p256]{Borel}, we have $H^ji_D^!D_{X,\alpha}=0$ for $j\neq 0,1$ ($\codim D=1$). 
Moreover, 
\begin{eqnarray*}
H^0i_D^!D_{X,\alpha}&=&\omega_{D/X}^{-1}\otimes_{\cO_D}J^{D_{X,\alpha}}, \mbox{ where }J^{D_{X,\alpha}}:=\{m\in D_{X,\alpha}| z_1m=0\}\cr
H^1i_D^!D_{X,\alpha}&=&D_X/\sum_{i=1}^nD_X(z_i\partial_i+ \alpha_i)+z_1D_X.\cr
\end{eqnarray*} 
Therefore $0=z_1\partial_1+\alpha_1=\alpha_1$ in $H^1i_D^!D_{X,\alpha}$, and $H^1i_D^!D_{X,\alpha}=0$.

Next we show that $J^{D_{X,\alpha}}=0$. Let $m\in D_{X,\alpha}$ with $z_1m=0$. Then
$$z_1m=h_1(z_1\partial_1+\alpha_1)+...+h_n(z_n\partial_n+\alpha_n)$$
for some $h_i\in D_X$. For the first factor $z_1$, we shall use the usual normal form $\sum_u p_u(z_1)\partial_1^u\in k[z_1]k[\partial_1]$ of a differential operator to represent an element of the Weyl algebra $D_{\bA^1}$. Then we can write uniquely
\begin{equation}
h_i=z_1g_i+r_i,~~i=1,..,n
\end{equation}
where $g_i\in D_X$ and summands in the normal form of $r_i$ do not involve $z_1$. 
We get
\begin{equation} \label{mid}
z_1(m-g_1(z_1\partial_1+\alpha_1)-...-g_n(z_n\partial_n+\alpha_n)-r_1\partial_1)=[r_1,z_1]\partial_1+\alpha_1r_1+r_2(z_2\partial_2+\alpha_2)+...+r_n(z_n\partial_n+\alpha_n)
\end{equation}
Observe that the normal form of the right side does not involve $z_1$. Thus both sides are zero, and we get
\begin{equation} \label{mid1}
m-g_1(z_1\partial_1+\alpha_1)-...-g_n(z_n\partial_n+\alpha_n)-r_1\partial_1=0
\end{equation}
and 
\begin{equation} \label{mid2}
[r_1,z_1]\partial_1+\alpha_1r_1+r_2(z_2\partial_2+\alpha_2)+...+r_n(z_n\partial_n+\alpha_n)=0
\end{equation}

The normal form of $r_1$ is expressed in the following finite sum 

\begin{equation} \label{mid3}
r_1=\sum_{j\geq 0} s_j\partial_1^j
\end{equation}
where (the normal form of) $s_j$ involves neither $z_1$ nor $\partial_1$.
Then \eqref{mid2} becomes
\begin{equation} \label{mid4}
-\sum_{j\geq 0} (\alpha_1+j)s_j\partial_1^j=r_2(z_2\partial_2+\alpha_2)+...+r_n(z_n\partial_n+\alpha_n).
\end{equation}
Now writing each $r_i$ in the normal form, i.e. as a sum of powers of $\partial_1$, and noting that $\partial_1$ commutes with $z_2\partial_2,..,z_n\partial_n$, it follows that the right side can be written uniquely in the form $\sum_j s_j'\partial_1^j$ where $s_j'\in \sum_{i=2}^n D_X(z_i\partial_i+\alpha_i)$. Given our assumptions on $\alpha_1$, \eqref{mid4} implies that for each $j\geq0$, 
$$s_j=-s_j'/(\alpha_1+j)\in \sum_{i=2}^n D_X(z_i\partial_i+\alpha_i)$$ 
By \eqref{mid3}, we get $r_1\in \sum_{i=2}^n D_X(z_i\partial_i+\alpha_i)$, hence $r_1\partial_1\in \sum_{i=2}^n D_X(z_i\partial_i+\alpha_i)$. Finally, by  \eqref{mid1}, we have $m\in \sum_{i=1}^n D_X(z_i\partial_i+\alpha_i)$, and thus $m\equiv 0$ in $D_{X,\alpha}$. So $J^{D_{X,\alpha}}=0$ and $H^0i_D^!D_{X,\alpha}=0$.

\end{proof}

We now return to the special case with $\beta=\beta_0$. 

\begin{cor}\label{DXbeta}
$D_{X,\beta_0}\simeq j_{!}j_{}^!D_{X,\beta_0}$.
\end{cor}

\begin{proof}
Since $\beta=\beta_0$, $\beta(\ft)=0$. Since $X$ is a smooth toric variety $X$, we can cover it by affine open toric subvarieties of the $\bA^r\times\bG_m^{s}$ corresponding to the cones in the fan of $X$. Since it suffices to show that the isomorphism holds on each such open set, we may as well assume that $X$ itself is an affine toric variety of this form, with the torus $T=\G_m^n$ acting on $X$ by scaling. Put $\alpha(e_i)=-\beta(e_i)+1$, where $\{e_i\}$ is the standard basis of $\ft\equiv k^n$. Then $\alpha$ satisfies condition (*). We  can now apply Proposition \ref{Kummer local system} on the affine toric variety $X$ and get
$$D_{X,\alpha}= j_+\calL_\alpha.$$
Taking Verdier dual yields
$D_{X,\beta}\simeq j_!\calL_{\beta}$.
\end{proof}

Next, we proceed to proving Theorem \ref{middle relative homology}.
Following \cite{HLZ}, we set $U_a:=X-Y_a$ ($Y_a\equiv V(f_a)$) and $\cF:=\Sol(D_{X,\beta})$. Restricting Corollary \ref{DXbeta} to $U_a$, taking $\Sol$, and noting that $\Sol f_!\simeq f_* \Sol$, $\Sol f^!\simeq f^*\Sol$ for any morphism $f$ and that $D_{X,\beta}|_{X-\cup D}\simeq\cO_{X-\cup D}$, we have
\begin{equation}\label{F restriction to constant sheaf}
\cF|_{U_a}\simeq j_{\cup D,*}\cF|_{U_a-\cup D}\simeq j_{\cup D,*}\bC|_{U_a-\cup D}[n].
\end{equation}

\begin{lem}\label{constant sheaf to relative homology}
Denote $p:U_a\rightarrow pt$. Then $R^n p_! j_{*}\bC|_{U_a-\cup D}$ is the relative homology\\ $H_n(U_a,U_a\cap(\cup D))$.
\end{lem}
\begin{proof}
Let $Y$ be a variety and $\omega_Y$ be the dualizing sheaf in constructive setting, so for $Y$ smooth, $\omega_Y=\bC[2\dim Y]$. Then $H_i(Y)=H^{-i}_c(\omega_Y)$.
Now if
$i: Z\subset Y$ is a closed subset and let $j:Y-Z\to Y$ be the complement, then we have
$$i_!\omega_Z\to \omega_Y\to Rj_*\omega_{Y-Z}\to $$
So $H_c^{-i}(Rj_*\omega_{Y-Z})= H_i(Y,Z)$.
Note that in our setting, $Y=U_a$ is smooth, $Z=U_a\cap (\cup D)$, and $j$ is an affine embedding. So 
$$R^np_!j_{*}\bC|_{U_a-\cup D}=H_c^{-n}(j_*\bC|_{U_a-\cup D}[2n])=H_n(U_a,U_a\cap (\cup D)).$$
\end{proof}

Now combining \eqref{F restriction to constant sheaf}, Lemma  \ref{constant sheaf to relative homology}, and \cite[Thm. 1.7]{HLZ}, we conclude the proof of our main result for $L=\omega_X^{-1}$ and $\beta=\beta_0$:

\begin{thm} (Chain integral solutions) \label{middle relative homology}
For any $a\in V^\vee$, we have canonical isomorphisms
\begin{equation*}
\Hom_{D_{V^\vee}}(\tau,\cO_{V^\vee,a}^{an})\simeq H^0_c(U_a,\cF|_{U_a})\simeq H_n(U_a,U_a\cap(\cup D)).
\end{equation*}
\end{thm}
This gives a new topological description of the classical solution space of the GKZ system $\tau$ in terms chains in the complements $U_a=X-Y_a$ that are bounded by the canonical divisor $\cup D$.

We note that the composition of the isomorphisms \cite{HLZ}
\begin{equation}\label{composition}
H_n(U_a,U_a\cap(\cup D)){\br\sim\over\ra} H^0_c(U_a,\cF|_{U_a}){\br\sim\over\ra} Hom_{D_{V^\vee}}(H^0\pi^\vee_+\cN,\cO_{V^\vee,a}^{an})
\end{equation}
is given by $C\mapsto (*\mapsto \left<C,*\right>)$, for $C\in H_n(U_a,U_a\cap(\cup D))$, where 
$$\pi_+^\vee\cN=\Omega_{U/V^\vee}^\bullet\otimes(\cO_{V^\vee}\boxtimes D_{X,\beta})[\dim X]|_U$$
and $\left<C,*\right>$ is the pairing between the chain $C$ with top forms.

Composing this with the isomorphism
\begin{equation}
\tau\simeq H^0\pi^\vee_+\cN,~~1\mapsto \frac{\Omega}{f}
\end{equation}
we have get the isomorphism
\begin{equation}
H_n(U_a,U_a\cap\cup D)\ra Hom_{D_{V^\vee}}(\tau,\cO_{V^\vee,a}^{an}),~~C\mapsto\left<C,*\frac{\Omega}{f}\right>_a.
\end{equation}
In particular, the chain $C$ corresponds to the function germ $\left<C,\frac{\Omega}{f}\right>_a$ as a local solution to $\tau$ at $a$. Therefore, the theorem shows that the space of local solution germs of $\tau$ at $a$ is exactly given by the chain integrals $\int_C{\Omega\over f_a}$.

\begin{cor}
For generic $a$, $\dim H_n(U_a,U_a\cap(\cup D))$ is equal to the volume of the polytope generated by the exponents of Laurent monomial basis $x^\mu$ of $V^\vee=\Gamma(X,\omega_X^{-1}).$
\end{cor}
\begin{proof}
Since the $\frakt$-character $\beta$ is a semi-nonresonant, the generic rank of the solution sheaf of $\tau$ is given by the volume of the polytope in question \cite{GKZ1990,Adolphson}. Now the corollary follows from \eqref{middle relative homology}.
\end{proof}

\begin{rmk}\hfill
\begin{itemize}
\item Theorem  \ref{middle relative homology} gives a topological interpretation of the GKZ's combinatorial volume formula for generic $a$ for generic rank of $\tau$ \cite{GKZ1990,Adolphson} on the one hand, but it holds for all $a$ on the other hand. The equation can also be viewed as the toric analogue of the statement that for any projective homogeneous variety $X$ \cite{HLZ}
\begin{equation*}
H^0_c(U_a,\cF|_{U_a})\simeq H_n(U_a,\C).
\end{equation*}

\item Under the identification of $\calF$ with $j_*\bC[n]$ above, the hypercohomology group of the perverse sheaf $H^0_c(U_a,\cF|_{U_a})$ in Theorem \ref{middle relative homology} inherits a MHS from the relative homology $H_n(U_a,U_a\cap(\cup D))$, in addition to providing the solution rank at each point $a$. This is anologous to the case of homogeneous varieties, whereby the hypercohomology inherits a mixed Hodge structure \cite{HLZ} from $H_n(U_a)$.

\item We point out that in principle we can carry out an explicit construction of chains in $U_a$ recursively starting from cycles. However, the construction is rather complicated combinatorially.
\end{itemize}
\end{rmk}

We now generalize Theorem \ref{middle relative homology} by replacing $L=\omega_X^{-1}$ with any very ample line bundle on the toric variety $X$, such that $L\otimes\omega_X$ is base point free. Let $\tau=\tau_{VW}$ now be the tautological system defined on $V^\vee\times W^\vee$ as in \cite[\S 6]{HLZ}. Then we have 

\begin{thm}\label{middle relative homology - general type}
For any $(a,b)\in V^\vee\times W^\vee$, we have canonical isomorphisms
\begin{equation*}
\Hom_{D_{V^\vee}}(\tau,\cO_{V^\vee\times W^\vee,(a,b)}^{an})\simeq H^0_c(U_a,\cF|_{U_a})\simeq H_n(U_a,U_a\cap(\cup D)).
\end{equation*}
\end{thm}
Note that the middle and the right side are both constant in the $W^\vee$ direction. The analogue of the chain integral map \eqref{middle relative homology} now becomes
$$
C\mapsto\int_C {g_b\Omega\over f_a}
$$
extending the cycle-to-period map of \cite{HLZ}. Note that the chain integrals are linear in $b$. The proof above carries over to this case verbatim.

\section{Concluding remarks}\label{concluding remarks}

We now comment on how the chain integral map might fail to be isomorphic. Let $X$ be a smooth complete toric variety of $\dim=n$, with the action of the torus $T$, and $G$ be an algebraic group acting on $X$ so that $T\subset G\subset Aut(X)$. Denote the corrsponding $\tau$ by $\tau^G$, which we shall study for various $G$. Note that the $G$-action on $X$ induces a stratification of $X$ by $G$-orbits, and denote $\cup D^G$ to be the union of $codim>0$ strata. We have a natural map

\begin{equation}\label{G cycle to T cycle}
r^G: H_n(U_a,U_a\cap(\cup D^G))\to H_n(U_a,U_a\cap(\cup D))
\end{equation}

Clearly, under the chain integral map \eqref{chain integral map}, $Ker (r^G)$ maps to 0. We now prove the following.

\begin{thm}\label{middle group}
The chain integral map induces an isomorphism
\begin{equation}
 H_n(U_a,U_a\cap(\cup D^G))/Ker (r^G) \simeq \Hom_{D_{V^\vee}}(\tau^G,\cO_{V^\vee,a}^{an})
\end{equation}
\end{thm}

\begin{proof}
 First note that there is the map $j_!\calO_{X-\cup D^G}\to D_{X,\beta}$ adjoint to $\calO_{X-\cup D^G}\simeq D_{X,\beta}|_{X-\cup D^G}$. Restricting to $U_a$ and taking $\Sol$ and $R^0p!$, it gives rise to a map
\[ H^0_c(U_a,\cF|_{U_a})\ra H_n(U_a,U_a\cap(\cup D^G))\]
which when $G=T$, is inverse to the first map in \ref{composition}, by corollary \ref{DXbeta}. One checks readily that the following diagram commutes:

$$
\begin{CD}
 H_n(U_a,U_a\cap(\cup D^G))/Ker (r^G)@>f_1^G>>\Hom_{D_{V^\vee}}(\tau^G,\cO_{V^\vee,a}^{an})@>\phi^G>>H^0_c(U_a,\cF^G|_{U_a})@>f_0^G>>\\
@V r^G VV@V i VV@V g VV \\ 
 H_n(U_a,U_a\cap(\cup D))@>f_1>>\Hom_{D_{V^\vee}}(\tau,\cO_{V^\vee,a}^{an})@>\phi>>H^0_c(U_a,\cF|_{U_a})@>f_0>>\\
\end{CD}
$$
$$
\begin{CD}
H_n(U_a,U_a\cap(\cup D^G))/Ker (r^G)@>f_1^G>>\Hom_{D_{V^\vee}}(\tau^G,\cO_{V^\vee,a}^{an})\\
@V r^G VV@V i VV\\
 H_n(U_a,U_a\cap(\cup D))@>f_1>>\Hom_{D_{V^\vee}}(\tau,\cO_{V^\vee,a}^{an})
\end{CD}
$$
where $f_1$ is the  chain integral map \eqref{chain integral map}, $f_0$ is the map coming from homological algebra at the beginning of the proof, $i$ is the obvious embedding, and $\phi$ is the canonical isomorphism as in \cite{HLZ}. The third square commutes due to the naturalness of adjoint functors.

The second row are all isomorphisms, and $f_0\phi f_1=Id$, $f_1f_0\phi=Id$, by what we have proved for $T$. Since $r^G$ and $i$ in the diagram are both injective, we deduce that $f_0^G\phi^Gf_1^G=Id$, and $f_1^Gf_0^G\phi^G=Id$, and therefore the first row are also all isomorphisms.
\end{proof}

Next we discuss a few other cases where we understand \eqref{chain integral map} more explicitly.

Case I, suppose $X=G/B$ is a flag variety, and we take the group $B$ in the definition of $\tau$ (therefore $D_{X,\beta}:=D_X\otimes_{U\frakb} k$). Then by the Beilinson-Bernstein localization, we have $D_{X,\beta}=i_{w_0,!}\calO_{X^{w_0}}$, where $i_{w_0}: X^{w_0}\into X$ is the inclusion of the open dense Schubert cell. So in this case, by the same argument as in the toric case, \eqref{chain integral map} is an isomorphism.

\begin{rmk}
Therefore, for $X=G/B$, the same argument as in the proof of Theorem \ref{middle group}, where one substitutes the toric $X$ with $G/B$, and the torus $T$ with $B$, shows that \ref{middle group} holds for $X$ and any parabolic subgroup $P$.

\end{rmk}

Case II, again take $X=G/B$, but take the maximal unipotent subgroup $N$ instead of $B$, then $D_{X,\beta}:=D_X\otimes_{U\frakn} k$ under Beilinson-Bernstein becomes a direct sum of Verma modules, of highest weights $-w(\rho)-\rho$, indexed by $w\in W$ the Weyl group of $G$, where $\rho$ is half the sum of positive roots. In other words, $D_{X,\beta}$ is isomorphic to the direct sum of $i_{w,!}\calO_{X^w}$, indexed by the Schubert cells $X^w$. So in this case, \eqref{chain integral map} is injective but not surjective in general. The extra solutions of $\tau$ come from lower dimensional "chain integral maps", associated with Schubert cells of higher codimensions.

As an explicit example of this case, take $X=\bP^1$ with homogeneous coordinates $[x:y]$, and let the unipotent subgroup $N$ be the 1-dimensional translation group, which leaves invariant $\infty=[1:0]$. Take a generic $a=a_1x^2+a_0xy+a_2y^2\in\Gamma(X,\cO(2))$. The extra solution of $\tau$ at $a$, that lies outside the chain integral map \eqref{chain integral map}, is the pairing of the 0-form on $x^2/a$ on $U_a$ with the zero cycle supported at $\infty$, which evaluates to $1/a_1$.

In connection to the chain integral map, we mention an old conjecture which seems intricately linked to it. In 1996, inspired by mirror symmetry, Hosono-Lian-Yau found a general combinatorial formula that gives a complete set of solutions to $\tau$ in the toric case. Their formula is a renormalized form of the formal GKZ Gamma series solution \cite{GKZ1990}. In fact, the formula gives an explicit {\it cohomology valued function \cite[eqn. (3.5)]{HLY1994}
\begin{equation}\label{B-series}
B_X:\calU_\infty\ra H^*(X^\vee,\C)
\end{equation}
such that the classical solution sheaf $\Hom_{D_{V^\vee}}(\tau,\cO_{V^\vee,a}^{an})$ is precisely generated by the functions $\int_\alpha B_X$ ($\alpha\in H_*(X^\vee,\Z)$).} Here $\calU_\infty$ is a neighborhood of the point $f_\infty=\zeta_1\cdots\zeta_p\in \Gamma(X,\calL)$ (which is a so-called large complex structure limit of the universal family $\cY$); the $\zeta_i=0$ are the defining equations of the irreducible $T$-invariant divisors in $X$. The space $X^\vee$ is a toric variety mirrored to $X$ in the sense of Batyrev. In addition as shown in \cite{HLY1994}, the fact that $B_X$ generates the solution sheaf of $\tau$ holds under the much weaker assumption that $X$ is semi-Fano toric. (There has also been a generalization of this solution formula recently to certain noncompact toric varieties by \cite{LLY2010} in the context of open string theory.) 

More importantly, based on an abundance of numerical evidence, it was conjectured that the period sheaf of the universal family of $\cY$ is generated precisely  by the functions 
$$
\int_\alpha B_X\cup [\cup D^\vee],\hskip.2in \alpha\in H_*(X^\vee,\Z)
$$
where $[\cup D^\vee]$ denotes the Poincar\'e dual of the canonical divisor $\cup D^\vee$ in $X^\vee$. This is the so-called {\it hyperplane conjecture}, which remains open. Note that hyperplane sections of $\cO_{X^\vee}(\cup D^\vee)$ are nothing but CY varieties mirrored to the CY varieties $Y_a$ in the family $\cY$.

At least intuitively, the parallel between the hyperplane and the chain integral conjectures seems striking. The statement \eqref{B-series} about $B_X$ above is clearly a combinatorial counterpart of the topological statement Theorem \ref{middle relative homology}. Under this dictionary, the hyperplane conjecture says that cupping $B_X$ with the (mirror anticanonical) class $[\cup D^\vee]$ corresponds to taking the subgroup of vanishing homology $H_{n-1}(Y_a)_{van}\into H_n(U_a,U_a\cap(\cup D))$, given by the `tube-over-cycle' map $\calT$. Note that the period sheaf of $\cY$ is generated precisely by the period integrals $\int_\gamma\Res{\Omega\over f_a}=\int_{\calT(\gamma)}{\Omega\over f_a}$.
To put it in another way, the groups $H_n(U_a,U_a\cap(\cup D))$ and $H^*(X^\vee,\C)$ are `mirror' to each other, while taking the subgroup of tubes over the cycles in $H_{n-1}(Y_a)_{van}$ of the CY $Y_a$ in $X$, should be mirror to passing to the quotient group $H^*(X^\vee,\C)/\Ann [\cup D^\vee]$ by the subgroup annihilated by the mirror CY divisor $\cup D^\vee$ in $X^\vee$.
This dictionary suggests a close connection between the hyperplane conjecture and the chain integral isomorphism: that taking tubes over cycles on the topological side may in fact corresponds to cupping with $[\cup D]$ on the cohomological side.

\vskip.5in

\noindent\address {\SMALL A. Huang, Department of Mathematics, Harvard University, Cambridge MA 02138. \\
\vskip-.45in  anhuang@math.harvard.edu.}
\vskip-.15in

\noindent\address {\SMALL B.H. Lian, Department of Mathematics, Brandeis University, Waltham MA 02454.\\
\vskip-.45in   lian@brandeis.edu.}
\vskip-.15in

\noindent\address {\SMALL S.-T. Yau, Department of Mathematics, Harvard University, Cambridge MA 02138. \\
\vskip-.45in  yau@math.harvard.edu.}
\vskip-.15in

\noindent\address  {\SMALL X. Zhu, Department of Mathematics, California Institute of Technology, Pasadena CA 91125. \\
\vskip-.45in xzhu@caltech.edu.}


\vfill
\begin{thebibliography}{10}


\bibitem{Adolphson} A. Adolphson, \emph{Hypergeometric Functions and Rings Generated by Monomials}, Duke Math. J. Vol 73, No. 2 (1994) 269-290.



\bibitem{ADJ} A.C. Avram, E. Derrick, and D. Jancic, \emph{On Semi-Periods}, arXiv:9511152.



\bibitem{BHLSY} S. Bloch, A. Huang, B.H. Lian, V. Srinivas, and S.-T. Yau, \emph{On the Holonomic Rank Problem}, J. Diff. Geom. 97 (2014), 11�35. arXiv:1302.4481v1.

\bibitem{Borel} A. Borel et al, \emph{Algebraic D-modules}, Academic Press 1987.










\bibitem{GKZ1990} I. Gel'fand, M. Kapranov and A. Zelevinsky, \emph{Hypergeometric Functions and Toral Manifolds},
English translation, Functional Anal. Appl. {\bf 23} (1989), 94-106.





\bibitem{HLY1994} S. Hosono, B.H. Lian and S-T. Yau, \emph{GKZ-generalized Hypergeometric Systems in Mirror Symmetry of Calabi-Yau hypersurfaces}, Commun. Math. Phys. {\bf 182} (1996), 535-577.


\bibitem{HLZ} A. Huang, B.H. Lian and X. Zhu, \emph{Period Integrals and the Riemann-Hilbert Correspondence}, arXiv:1303.2560

\bibitem{JS2009} H. Jockers, M. Soroush, \emph{Effective superpotentials for compact D5-brane Calabi-Yau
geometries},  Comm.Math.Phys. 290 (2009), No.1, 249-290, arXiv:hep-th/0808.0761.


\bibitem{Kapranov1997} M. Kapranov, \emph{Hypergeometric Functions on Reductive Groups}, Integrable Systems and Algebraic Geometry (Kobe/Kyoto, 1997), 236-281, World Sci. Publ., River Edge, NJ, 1998.





\bibitem{KW2008}D. Krefl, J. Walcher, \emph{Real Mirror Symmetry for One-Parameter Hypersurfaces},
 J.High Energy Phys. 2008, no.9, 031, arXiv:hep-th/0805.0792



\bibitem{LLY2010} S. Li, B.H. Lian and S-T. Yau, \emph{Picard-Fuchs Equations for Relative Periods and Abel-Jacobi Map for Calabi-Yau Hypersurfaces}, Amer. J. Math. Vol. 134, No. 5 (2012) 1345-1384.

\bibitem{LSY} B.H. Lian, R. Song and S.-T. Yau, \emph{Period Integrals and Tautological Systems}, Journ. EMS Vol. 15, 4 (2013) 1457-1483. arXiv:1105.2984v3

\bibitem{LY} B.H. Lian and S.-T. Yau, \emph{Period Integrals of CY and General Type Complete Intersections}, Invent. Math. Vol 191, 1 (2013) 35-89. arXiv:1105.4872v3.








\end{thebibliography}
\end{document}